\documentclass[10pt,leqno]{amsart}

\usepackage[T1]{fontenc}
\usepackage[utf8]{inputenc}
\usepackage[english]{babel}
\usepackage[left=1.1in,right=1.1in,top=1.5in,bottom=1.5in]{geometry}
\usepackage[foot]{amsaddr}
\usepackage{amssymb}
\usepackage{amsthm}
\usepackage{amsmath}
\usepackage{amsfonts}
\usepackage{mathtools}
\usepackage{enumitem}
\usepackage[dvipsnames]{xcolor}
\usepackage{tikz}
\usetikzlibrary{arrows.meta,calc,decorations.markings}

\definecolor{dkblue}{RGB}{30,90,140}
\definecolor{outflowblue}{RGB}{32,94,151}
\definecolor{outflowred}{RGB}{181,62,55}
\definecolor{outflowgreen}{RGB}{39,126,103}
\usepackage[
  colorlinks=true,
  pdfstartview=Fit, 
  linkcolor=dkblue,
  citecolor=dkblue,
  urlcolor=dkblue
]{hyperref}

\theoremstyle{plain}
\newtheorem{theorem}{Theorem}[section]
\newtheorem{proposition}[theorem]{Proposition}
\newtheorem{lemma}[theorem]{Lemma}

\theoremstyle{definition}
\newtheorem{definition}[theorem]{Definition}

\theoremstyle{remark}
\newtheorem{remark}[theorem]{Remark}

\numberwithin{equation}{section}

\newcommand{\R}{\mathbb R}
\newcommand{\B}{\mathcal B}
\newcommand{\Om}{\Omega}
\newcommand{\obs}{\mathcal K}
\newcommand{\er}{e_r}
\newcommand{\etheta}{e_\theta}
\newcommand{\dd}{\,\mathrm d}

\DeclareMathOperator{\dist}{dist}
\DeclareMathOperator{\curl}{curl}
\DeclareMathOperator{\diver}{div}
\DeclareMathOperator{\sign}{sign}

\begin{document}
\allowdisplaybreaks

\title[Rigidity results for steady Euler flows in two dimensional exterior domains]
{Rigidity results for Euler flows in two dimensional exterior domains}

\author[D. Alonso-Or\'an]{Diego Alonso-Or\'an}
\address{Departamento de An\'{a}lisis Matem\'{a}tico and Instituto de
Matem\'{a}ticas y Aplicaciones (IMAULL), Universidad de La Laguna,
C/Astrof\'{i}sico Francisco S\'{a}nchez s/n, 38271 La Laguna, Spain}
\email{\href{mailto:dalonsoo@ull.edu.es}{dalonsoo@ull.edu.es}}

\author[D. S\'anchez-Sim\'on del Pino]{Daniel S\'anchez-Sim\'on del Pino}
\address{Institute for Applied Mathematics, University of Bonn,
Endenicher Allee 60, D-53115 Bonn, Germany}
\email{\href{mailto:sanchez@iam.uni-bonn.de}{sanchez@iam.uni-bonn.de}}

\author[J. J. L. Vel\'azquez]{J. J. L. Vel\'azquez}
\address{Institute for Applied Mathematics, University of Bonn,
Endenicher Allee 60, D-53115 Bonn, Germany}
\email{\href{mailto:velazquez@iam.uni-bonn.de}{velazquez@iam.uni-bonn.de}}

\date{\today}

\subjclass[2020]{Primary 35Q31}
\keywords{Steady Euler equations, exterior domains, stagnation points}

\begin{abstract}
We study steady solutions of the two dimensional incompressible Euler
equations in the exterior of a compact domain. We assume that fluid enters the exterior domain through every point of the obstacle boundary. Under suitable assumptions on the direction
and decay of the velocity at infinity, we prove that the flow is
irrotational if and only if the velocity has no stagnation points.
\end{abstract}

\maketitle
\tableofcontents

\section{Introduction}\label{sec:introduction}

In this paper we study steady solutions of the incompressible Euler
equations in the exterior of a bounded two-dimensional domain. Let
$\obs\subset\mathbb R^2$ be a bounded domain whose boundary is a regular
$C^2$ Jordan curve, and set
$\Om:=\mathbb R^2\setminus\obs$. The set $\obs$
represents an obstacle and $\Om$ is the region occupied by the fluid.
 We consider a velocity field
$v\in C^1(\overline\Om;\mathbb R^2)$ and a
pressure $p\in C^1(\Om)$ satisfying
\begin{equation}\label{eq:euler-intro}
  (v\cdot\nabla)v+\nabla p=0,
  \qquad
  \diver v=0
  \qquad\textup{in }\Om.
\end{equation}
The equations are understood pointwise. We denote by $N$ the unit normal
to $\partial\obs$ pointing into $\Om$ and assume that
\begin{equation}\label{eq:source-boundary-intro}
  v\cdot N>0
  \qquad\textup{on }\partial\obs.
\end{equation}

We also impose conditions on the velocity at infinity. Writing
$v=v_r\er+v_\theta\etheta$ in polar coordinates, we assume that
$v_r>0$ and $|v|^2/v_r\leq C/|x|$ outside a sufficiently large disk.
These assumptions are stated precisely in \eqref{eq:far-field-cost}.
The first condition means that the distance from the origin increases
along a trajectory whenever it lies outside this disk. The second
condition controls the circulation along such trajectories and implies
$|v(x)|\leq C/|x|$. It holds, for example, if
$0<v_r\leq M/|x|$ and $|v_\theta|\leq K v_r$ there, for some
$M,K>0$. A natural direction for future work is to identify the weakest conditions at infinity under which the absence of stagnation points forces the fluid to be irrotational. \medskip

Our main result states that, under these assumptions, the velocity does
not vanish in $\overline\Om$ if and only if the flow is irrotational.
In particular, if the vorticity is not identically zero, there must be a
point in $\Om$ where the velocity vanishes. Such a point is called a
stagnation point. The conclusion does not require radial symmetry of the
flow or of the obstacle.

\subsection{Relation with magnetohydrostatics}\label{subsec:mhs-motivation}

The problem can also be written in terms of the equations of ideal
magnetohydrostatics. A static magnetic field $B$ and a plasma pressure
$p_{\mathrm m}$ satisfy
\begin{equation}\label{eq:mhs-intro}
  (\curl B)\times B=\nabla p_{\mathrm m},
  \qquad
  \diver B=0.
\end{equation}
The identity
\[
  (\curl B)\times B
  =(B\cdot\nabla)B-\nabla\left(\frac12|B|^2\right),
\]
shows that \eqref{eq:mhs-intro} is equivalent to the steady Euler
equations under the identification
\begin{equation}\label{eq:euler-mhs-identification}
  v=B,
  \qquad
  p=-p_{\mathrm m}-\frac12|B|^2.
\end{equation}
For a two dimensional field $B=(B_1,B_2,0)$,
the scalar vorticity corresponds to the third component of $\curl B$.
A stagnation point corresponds to a zero of the magnetic field. Our
theorem therefore implies that, under the stated boundary condition and
assumptions at infinity, a two dimensional equilibrium with nonzero current must
have a point where the magnetic field vanishes. \medskip

The relation between the two systems also connects our problem with the boundary value problems proposed by Grad and Rubin
\cite{GradRubin1958}. These problems prescribe the normal component of
the magnetic field on the boundary, together with additional data on the
part of the boundary through which the field enters the domain. Depending
on the formulation, these additional data concern the current, the
pressure or the tangential component of the magnetic field. The equations
then determine quantities along the magnetic field lines. This leads to
a combination of transport equations for the vorticity $\omega$ and a div-curl system for the velocity $v$. \medskip

Existence results for magnetohydrostatic equilibria were obtained
by Lortz \cite{Lortz1970}. In the context of steady Euler flows, Alber
\cite{Alber1992} used transport of the vorticity to construct three
dimensional solutions with non-vanishing vorticity. More recently, problems of the type proposed by
Grad and Rubin have been studied in two dimensions
\cite{AlonsoOranVelazquez2024}, in three dimensions for perturbations of
a reference configuration
\cite{AlonsoOranSanchezSimonVelazquez2026}, and in domains that include
perturbations of annuli and spherical shells
\cite{SanchezSimonVelazquez2026}. These results give existence of
equilibria whose field lines connect different parts of the boundary.
In our exterior problem, the transport equation instead leads to a
restriction. If the field does not vanish, the current must be zero. \medskip

\subsection{Previous rigidity results}\label{subsec:previous-results}

The relation between stagnation points and the structure of steady Euler
flows has been studied in several domains. Hamel and Nadirashvili
\cite{HamelNadirashvili2017} proved that a steady flow in a strip is a
shear flow if the velocity is tangent to the boundary and its magnitude
is bounded away from zero. They also obtained a corresponding result in
a half plane under an additional boundedness assumption on the velocity.
In the whole plane, they proved that a bounded steady flow whose speed
is bounded away from zero must be a parallel shear flow
\cite{HamelNadirashvili2019}. Their arguments use the geometry of the
trajectories and elliptic equations for the stream function.

The assumption that the speed is bounded away from zero is important.
De Regibus and Ruiz \cite{DeRegibusRuiz2025} constructed bounded steady
Euler flows in the whole plane which have no stagnation points and are
not shear flows. Thus the absence of zeros alone does not imply the
conclusion of the preceding result. In our problem the velocity tends
to zero at infinity. The boundary condition and the additional control
of its radial component provide the information used in the proof. \medskip

Hamel and Nadirashvili also studied bounded annuli, exterior circular
domains and punctured domains \cite{HamelNadirashvili2023}. They assume that the velocity is tangent to the boundary and does
not vanish. In unbounded domains they impose additional conditions
at infinity, and in punctured domains they impose additional
conditions near the puncture. Under these assumptions, they prove
that the streamlines are concentric circles. The proof uses trajectories of the velocity, trajectories of
the gradient of the stream function and the moving plane method for a
semilinear elliptic equation. Wang and Zhan \cite{WangZhan2024} proved
circular symmetry in disks with exactly one interior stagnation point
and in annuli without interior stagnation points. Cao, Fan and Zhan
\cite{CaoFanZhan2024} studied related free boundary problems in exterior
domains. Under their assumptions, a flow tangent to the boundary, with
constant boundary speed and no interior stagnation points, must be
circular and the obstacle must be a disk.\medskip

Drivas and Nualart \cite{DrivasNualart2026} considered straight periodic
channels and circular annuli with velocity tangent to the boundary. They showed
that, if every streamline is a closed curve which cannot be contracted
to a point in the domain, then the flow is parallel or
circular, respectively. In our problem the flux through the boundary
of the obstacle is positive and the trajectories under consideration
go to infinity. This flux prevents the existence of a stream function
defined on all of $\Om$. The reduction to a semilinear elliptic equation
for a global stream function is therefore not directly available.\medskip

Constantin, Drivas and Ginsberg \cite{ConstantinDrivasGinsberg2021}
proved symmetry results for suitable steady Euler, Boussinesq and
magnetohydrostatic flows in periodic channels and for axisymmetric
configurations in cylinders. They also constructed families of steady
states by deforming suitable reference solutions and their domains.
Elgindi, Huang, Said and Xie \cite{ElgindiHuangSaidXie2026} studied
analytic steady Euler flows in bounded simply connected planar domains
with analytic boundary. When the velocity is tangent to the boundary, they
proved that the stream function is either radial or satisfies a
semilinear elliptic equation of the form $\Delta\psi=F(\psi)$.
For smooth flows, however, the vorticity need not be a function of the
stream function on the whole domain. Elgindi and Huang
\cite{ElgindiHuang2026} constructed examples with this property in
which all critical points of the stream function are nondegenerate.
 \medskip

Gui, Xie and Xu \cite{GuiXieXu2026} obtained a classification in terms
of the directions of the velocity. For a bounded steady flow in the
whole plane, these directions cover the unit circle unless the flow is
a parallel shear flow. For irrotational flows, Hamel and Karakhanyan
\cite{HamelKarakhanyan2025} proved that a potential flow in an infinite
cylinder is constant and parallel to its axis if the velocity is tangent to the boundary and the speed is bounded away from zero. Related results
on symmetry and on the existence of solutions without radial symmetry
involve compactly supported flows and flows in annular domains
\cite{Ruiz2023,EncisoFernandezRuizSicbaldi2025}.\medskip

Other works allow stagnation points or regions where the velocity
vanishes. Li, L\"u, Shahgholian and Xie
\cite{LiLuShahgholianXie2026} studied existence and uniqueness of steady
flows in infinitely long nozzles and the regularity of the boundaries
of such regions. De Regibus, Esposito and Ruiz
\cite{DeRegibusEspositoRuiz2026} proved that every classical steady
Euler flow with finite energy in the whole plane has a stagnation
point. They also obtained classification results under additional
assumptions on the set of stagnation points and on the vorticity.
The positive flux in our problem forces the energy to be infinite,
as shown in Remark~\ref{rem:critical-energy}.\medskip

Huang, Xie and Zhang \cite{HuangXieZhang2026} recently studied steady
flows in an infinite channel with prescribed shear flows as limits at
infinity. The speeds of these limiting flows are bounded away from
zero. For a broad class of prescribed profiles, they constructed smooth
steady solutions which are not shear flows. For a dense family of
analytic profiles, they proved that every analytic steady state with
the prescribed limits is a shear flow. \medskip

This paper then provides a rigidity result in the sense that it restricts class of boundary value problems that can be solved \eqref{eq:euler-intro} in an exterior domain. For instance, it is not possible to extend the results in \cite{AlonsoOranSanchezSimonVelazquez2026,AlonsoOranVelazquez2024} to an unbounded domain. Indeed, in these articles, if one introduces a small enough perturbation so that there are no stagnation points, the resulting flow has to be potential. This then restricts the possible boundary values. The main differences in the present problem are the positive flux
through the obstacle boundary and the decay of the velocity at infinity. We do not assume a positive uniform lower bound on the speed or radial symmetry of the obstacle. We prove that the absence of
stagnation points forces the vorticity to vanish. The techniques here are also substantially different from the ones used in previous rigidity results, since the decay of our velocity field prevents us to define a single valued stream function. Therefore, the elliptic PDE techniques in previous papers do not apply. 

\subsection{Main result}\label{subsec:contributions}
Let $\obs\subset\R^{2}$ be a compact domain, with $0\in \obs$, whose boundary is a regular $C^{2}$ Jordan curve. Define $\Om:=\R^2\setminus\obs$.
The normal $N$ points from $\obs$ into $\Om$, so the
outward unit normal of $\Om$ on $\partial\obs$ is $-N$. For $x\neq0$,
we write
\[
  r=|x|,
  \qquad
  \er(x)=\frac{x}{|x|},
  \qquad
  \etheta(x)=\er(x)^\perp,
  \qquad
  (y_1,y_2)^\perp=(-y_2,y_1).
\]

We collect the assumptions used in the theorem. The velocity and
pressure satisfy
\begin{equation}\label{eq:regularity}
  v\in C^1(\overline\Om;\mathbb R^2),
  \qquad
  p\in C^1(\Om),
\end{equation}
and solve \eqref{eq:euler-intro} pointwise. At the boundary we assume
\begin{equation}\label{eq:source-boundary}
  v\cdot N>0
  \qquad\textup{on }\partial\obs.
\end{equation}
At infinity we assume that there exist $R_0,C>0$ such that
\begin{equation}\label{eq:far-field-cost}
  v_r(x):=v(x)\cdot\er(x)>0,
  \qquad
  \frac{|v(x)|^2}{v_r(x)}\leq\frac{C}{|x|}
  \qquad\textup{for }|x|\geq R_0.
\end{equation}
Increasing $R_0$ if necessary, we may assume, without loss of generality, that
\begin{equation}\label{eq:obstacle-inside-R0}
  \obs\subset B_{R_0}.
\end{equation}
Since $0<v_r\leq|v|$, condition \eqref{eq:far-field-cost} gives
\begin{equation}\label{eq:speed-upper}
  |v(x)|\leq\frac{C}{|x|}
  \qquad\textup{for }|x|\geq R_0.
\end{equation}
Indeed, $|v|^2\leq(C/|x|)v_r\leq(C/|x|)|v|$, and $|v|>0$ in this
region. \medskip

A sufficient condition for \eqref{eq:far-field-cost} is that there
exist $M,K>0$ such that
\begin{equation}\label{eq:far-field-cone}
  0<v_r(x)\leq\frac{M}{|x|},
  \qquad
  |v_\theta(x)|\leq K v_r(x)
  \qquad\textup{for }|x|\geq R_0.
\end{equation}
In fact,
\begin{equation}\label{eq:cone-implies-cost}
  \frac{|v|^2}{v_r}
  =v_r+\frac{v_\theta^2}{v_r}
  \leq(1+K^2)v_r
  \leq\frac{(1+K^2)M}{|x|}.
\end{equation}

We denote the scalar vorticity and the Bernoulli function by
\begin{equation}\label{eq:omega-B}
  \omega:=\partial_1v_2-\partial_2v_1,
  \qquad
  \B:=p+\frac12|v|^2.
\end{equation}
The Euler equations imply
\begin{equation}\label{eq:euler-identities}
  \diver(\omega v)=0\quad\textup{in }\mathcal D'(\Om),
  \qquad
  \nabla\B=-\omega v^\perp\quad\textup{in }\Om.
\end{equation}
The first identity follows by taking the scalar curl of the momentum
equation in the sense of distributions. The second holds pointwise
and follows from
$(v\cdot\nabla)v=\nabla(|v|^2/2)+\omega v^\perp$. \medskip

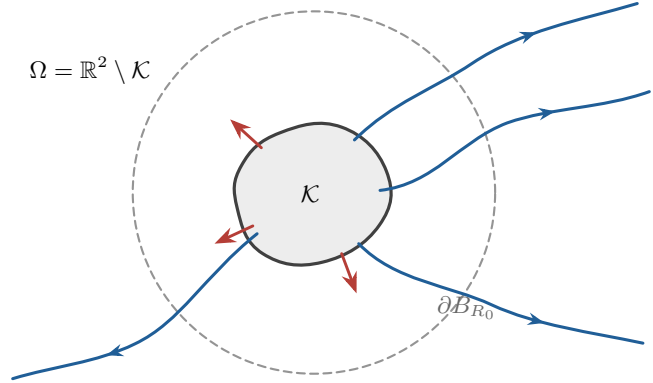
\begin{figure}[t]
  \centering
  \begin{tikzpicture}[
      scale=0.94,
      >=Stealth,
      line cap=round,
      line join=round,
      obstacle/.style={draw=black!75,very thick,fill=black!7},
      farcircle/.style={draw=black!42,densely dashed,line width=0.8pt},
      streamline/.style={draw=outflowblue,line width=1.15pt,
        postaction={decorate},
        decoration={markings,mark=at position 0.67 with
          {\arrow{Stealth[length=2.2mm]}}}},
      boundaryarrow/.style={->,draw=outflowred,line width=1.1pt}
    ]
    \draw[farcircle] (0,0) circle (2.55);
    \node[font=\small,black!60]
      at (2.15,-1.62) {$\partial B_{R_0}$};

    \path[obstacle]
      plot[smooth cycle,tension=0.75] coordinates
      {(-1.08,-0.22) (-0.72,-0.86) (0.12,-0.98)
       (0.92,-0.46) (1.02,0.30) (0.39,0.92)
       (-0.48,0.84) (-1.02,0.38)};
    \node[font=\small] at (-0.05,-0.02) {$\obs $};
    \node[font=\small] at (-3.15,1.72)
      {$\Omega=\mathbb R^2\setminus{\mathcal K}$};

    \draw[boundaryarrow] (-0.74,0.63) -- (-1.18,1.03);
    \draw[boundaryarrow] (-0.86,-0.47) -- (-1.40,-0.72);
    \draw[boundaryarrow] (0.39,-0.86) -- (0.60,-1.42);

    \draw[streamline]
      (0.93,0.03) .. controls (1.70,0.12) and (2.04,0.72) ..
      (2.72,0.98) .. controls (3.35,1.20) and (4.05,1.18) .. (4.72,1.42);
    \draw[streamline]
      (0.57,0.74) .. controls (1.12,1.28) and (1.68,1.45) ..
      (2.23,1.82) .. controls (2.92,2.28) and (3.73,2.42) .. (4.55,2.66);
    \draw[streamline]
      (0.63,-0.72) .. controls (1.14,-1.25) and (1.90,-1.30) ..
      (2.50,-1.58) .. controls (3.13,-1.87) and (3.87,-1.96) .. (4.63,-2.12);
    \draw[streamline]
      (-0.80,-0.58) .. controls (-1.55,-1.18) and (-1.87,-1.82) ..
      (-2.52,-2.13) .. controls (-3.07,-2.40) and (-3.65,-2.44) .. (-4.24,-2.62);
  \end{tikzpicture}
  \caption{Geometry of the exterior flow. The red arrows indicate fluid
  entering the exterior domain through the boundary. Outside
  the dashed circle, the radial component of the velocity is positive.
  The blue curves illustrate trajectories that tend to infinity.}
  \label{fig:exterior-outflow-geometry}
\end{figure} 

We can now state the main result.

\begin{theorem}\label{thm:main}
Let $\obs$, $\Om$ and $(v,p)$ satisfy 
\eqref{eq:euler-intro}, \eqref{eq:regularity}, \eqref{eq:source-boundary} and
\eqref{eq:far-field-cost}. Then the following statements are equivalent:
\begin{enumerate}[label=\textup{(\roman*)}]
  \item\label{itemi} $v(x)\neq0$ for every $x\in\overline\Om$.
  \item \label{itemii}$\omega\equiv0$ in $\Om$.
\end{enumerate}
\end{theorem}

Since \eqref{eq:source-boundary} excludes zeros on $\partial\obs$, any
stagnation point lies in $\Om$. 
Figure~\ref{fig:exterior-outflow-geometry} illustrates the geometry.

\subsection{Strategy of the proof and remarks}\label{subsec:proof-strategy}
We briefly describe the proof of Theorem~\ref{thm:main}.
We first prove implication \ref{itemii} $\Rightarrow$ \ref{itemi}
by the argument principle. If the flow is
irrotational, incompressibility implies that the complex velocity
$q=v_1-iv_2$ is holomorphic. The condition $v\cdot N>0$ determines
the winding number of $q$ along the boundary of the obstacle.
The condition $v_r>0$ gives the same winding number along every
sufficiently large circle, when both curves are oriented
counterclockwise. The argument principle then shows that $q$ has
no zeros in $\Om$. Since the boundary condition also excludes zeros
on $\partial\obs$, the velocity does not vanish in $\overline\Om$.
For this implication we only need the positivity of $v_r$ at infinity,
without the bound on $|v|^2/v_r$.

For implication \ref{itemi} $\Rightarrow$ \ref{itemii}, we assume
that $v$ does not vanish in $\overline\Om$. We first show that
every trajectory tends to infinity in forward time. In two dimensions,
the vorticity of a steady Euler flow is constant along trajectories.
Therefore, if the vorticity is nonzero at some point, we can construct
a tube of trajectories on which the vorticity has a fixed sign and
remains bounded away from zero. \medskip

We now estimate the integral of the vorticity over this tube in two
different ways. The flux through each circular section of the tube
is positive and independent of the radius. Combining this fact with
the bound $|v(x)|\leq C/|x|$, we obtain a quadratic lower bound for
the area of the tube truncated at radius $R$. Since the vorticity
has a fixed sign and remains bounded away from zero, the absolute
value of its integral over the tube also has a quadratic lower bound.
On the other hand, Stokes' theorem allows us to estimate this integral
using the values of the velocity on the boundary of the tube.
The bound on $|v|^2/v_r$ shows that the circulation around the tube grows at most logarithmically with $R$.
Comparing these two estimates and letting $R\to\infty$,
we obtain a contradiction. \medskip

\begin{remark}\label{rem:source-condition}
Condition \eqref{eq:source-boundary} means that fluid enters $\Om$ through
the boundary of the obstacle. Indeed, the outward normal of $\Om$ there
is $-N$, and $v\cdot(-N)<0$. By continuity and compactness,
$v\cdot N\geq c_0>0$ on $\partial\obs$ for some $c_0$. It follows that
the velocity has no zeros on the boundary. 
\end{remark}

\begin{remark}\label{rem:critical-energy}
The positive flux through the obstacle boundary implies that the
energy is infinite. To see this, let
\[
  \Phi:=\int_{\partial\obs}v\cdot N\dd\ell>0.
\]
For every $r\geq R_0$, the divergence theorem in
$B_r\setminus\obs$ gives
\begin{equation}\label{eq:source-flux-circles}
  \Phi=\int_{\partial B_r}v_r\dd\ell.
\end{equation}
Applying Cauchy's inequality on $\partial B_r$, we obtain
\begin{equation}\label{eq:circle-energy-lower}
  \int_{\partial B_r}|v|^2\dd\ell
  \geq\frac{\Phi^2}{2\pi r}.
\end{equation}
Hence, for $R>R_0$,
\[
  \int_{B_R\setminus B_{R_0}}|v|^2\dd x
  \geq\frac{\Phi^2}{2\pi}\log\frac{R}{R_0}.
\]
In particular, $\int_\Om|v|^2\dd x=\infty$. The results for flows
with finite energy do not apply in this setting.
\end{remark}

\begin{remark}
A familiar sufficient condition for \eqref{eq:far-field-cost} is the
asymptotic behavior
\begin{equation}\label{eq:physical-profile}
  v(x)=\frac{a}{|x|}\er(x)+\frac{b}{|x|}\etheta(x)
  +o(|x|^{-1})
  \qquad\textup{as }|x|\to\infty,
\end{equation}
where $a>0$, $b\in\mathbb R$.
Indeed,
\[
  v_r=\frac{a}{|x|}+o(|x|^{-1}),
  \qquad
  \frac{|v|^2}{v_r}
  =\frac{a^2+b^2}{a|x|}+o(|x|^{-1}),
\quad \textup{as }|x|\to\infty.\]
These relations imply \eqref{eq:far-field-cost} after increasing $R_0$
and choosing $C$ sufficiently large.
\end{remark}

\begin{remark}\label{rem:rigid-wall}
If the velocity is tangent to the boundary of the obstacle, a flow
without stagnation points need not be irrotational. For example, in the exterior of the
disk $B_a$, let
\[
  v(r,\theta)=V(r)\etheta,
\]
where $V\in C^1([a,\infty))$ is everywhere positive or everywhere
negative. This field has no zeros and solves the steady Euler equations
with a radial pressure satisfying $p'(r)=V(r)^2/r$. Its vorticity is
$\omega=V'(r)+V(r)/r$, which need not vanish. 
These circular flows do not satisfy $v_r>0$. More generally, an
incompressible flow tangent to the obstacle boundary has zero flux
through every circle enclosing the obstacle. It therefore cannot have
$v_r>0$ on any such circle. This example concerns a different boundary
condition and does not satisfy the assumptions of
Theorem~\ref{thm:main}. 
\end{remark}

\subsection{Notation and preliminaries}\label{subsec:notation}

For $R>0$, we write
$B_R:=\{x\in\mathbb R^2:|x|<R\}$. Throughout the paper,
$C^1(\overline\Om)$ means $C^1$ regularity in $\Om$ with the function
and its first derivatives extending continuously to $\overline\Om$.
This notation does not impose boundedness at infinity. We write $\mathcal D'(U)$
for the space of distributions on an open set $U$. If $f$ is a
holomorphic function which is not identically zero, we denote the order
of a zero at $z$ by $\operatorname{ord}_z(f)$.

We identify $\mathbb R^2$ with $\mathbb C$ by writing $z=x_1+ix_2$.
For a vector field $v=(v_1,v_2)$, its complex velocity is
\begin{equation}\label{eq:complex-velocity}
  q(z):=v_1(z)-iv_2(z).
\end{equation}
We use the Wirtinger derivatives
\[
  \partial_z:=\frac12(\partial_{x_1}-i\partial_{x_2}),
  \qquad
  \partial_{\bar z}:=\frac12(\partial_{x_1}+i\partial_{x_2}).
\]
A direct computation gives
\begin{equation}
  2\partial_{\bar z}q
=\left(\partial_{x_1}+i\partial_{x_2}\right)(v_1-iv_2)=\partial_{x_1}v_1+\partial_{x_2}v_2
    -i\left(\partial_{x_1}v_2-\partial_{x_2}v_1\right).
\end{equation}
Therefore,
\begin{equation}\label{eq:dbar-q}
  2\partial_{\bar z}q=\diver v-i\omega.
\end{equation}
In particular, $q$ is holomorphic if $\diver v=0$ and $\omega=0$.
If $\mathsf v:=v_1+iv_2$ is the usual complex representation of the
vector field, then $q=\overline{\mathsf v}$. Thus $q(z)=0$ if and
only if $v(z)=0$.

We recall some facts about closed curves. A Jordan curve is a curve
$\Gamma$ homeomorphic to $\mathbb S^1$. The Jordan curve theorem states
that $\R^2\setminus\Gamma$ has two connected components $U_1$ and $U_2$.
Both are open, with $U_1$ bounded and $U_2$ unbounded. By the Schoenflies
theorem, $\overline{U_1}$ is homeomorphic to the closed unit disk,
$U_1$ to $B_1$, and $U_2$ to $\R^2\setminus\{0\}$.
We now recall the definition of the degree of a circle map.
See \cite{DoCarmo} for the geometric background.

A parametrization of a Jordan curve $\Gamma$ is a continuous map $\gamma:[0,2\pi]\rightarrow\Gamma$ such that $\gamma|_{[0,2\pi)}$ is injective and $\gamma(0)=\gamma(2\pi)$.
\begin{definition}\label{def:circle-degree}
    Let $f:\mathbb{S}^1\rightarrow \mathbb{S}^1$ be a continuous map. Let $F:\R\rightarrow\mathbb R$ be a continuous lift of $f$, i.e., a continuous map such that
    $f(e^{it})=e^{iF(t)}$ for every $t\in\R$. We then define the degree of $f$ as
    $$\deg(f):=\frac{1}{2\pi}\left(F(2\pi)-F(0)\right).$$
\end{definition}

We now recall the basic properties of the degree. The proof of Proposition~\ref{prop:degree-properties} can be found in
\cite[Sections~1.1 and~2.2]{Hatcher2002}.
\begin{proposition}\label{prop:degree-properties}    
Let $f:\mathbb{S}^1\rightarrow\mathbb{S}^1$ be continuous. The following properties hold: 
    \begin{itemize}
        \item $\deg(f)$ is well defined, i.e., it does not depend on the choice of the lift $F$. 
        \item $\deg(f)\in \mathbb{Z}$.
        \item For every continuous  map $g:\mathbb{S}^1\rightarrow\mathbb{S}^1$, we have that
        $$\deg(f\circ g)=\deg(f)\deg(g).$$
    \end{itemize}
\end{proposition}

We can use the notion of degree to define the winding number of a curve as follows. 

\begin{definition}
    Let $c:[0,2\pi]\rightarrow\mathbb{C}\setminus\{0\}$ be a continuous closed curve. We define the winding number of $c$ as 
    $$\operatorname{wind}(c,0)=\deg\left(\frac{c}{|c|}\right).$$
    Here the endpoints of $[0,2\pi]$ are identified.
    Furthermore, if $\Gamma$ is an oriented Jordan curve and $u:\Gamma\rightarrow\mathbb{R}^2\setminus\{0\}$ is continuous, we define 
    $$\deg_\Gamma (u):=\operatorname{wind}(\textsf{u}\circ \gamma,0),$$
    where $\gamma$ is any parametrization $\gamma:[0,2\pi]\rightarrow \Gamma$ and $\textsf{u}:=u_1+iu_2.$ 
\end{definition}

We recall the following standard properties of the winding number (see \cite[Section~2.2]{Hatcher2002} and \cite[Chapter~4]{Ahlfors1979}).
\begin{lemma}\label{lem:winding-properties}
The following properties hold.
\begin{enumerate}[label=\textup{(\roman*)}]
    \item $\deg_\Gamma (u)$ is well defined, i.e. it does not depend on the choice of the parametrization $\gamma$, provided $\gamma$ preserves the orientation.
  \item
  If $H:[0,2\pi]\times[0,2\pi]\to\mathbb C\setminus\{0\}$ is continuous
  and $H(s,0)=H(s,2\pi)$ for every $s\in[0,2\pi]$, then
  $\operatorname{wind}(H(s,\cdot),0)$ is independent of $s$.

  \item
  If $c$ is a continuous closed curve in
  $\mathbb C\setminus\{0\}$, then
  \[
    \operatorname{wind}(\overline c,0)
    =
    -\operatorname{wind}(c,0), \ \text{ where } \overline{c} \text{ denotes the complex conjugate.}
  \]

  \item Given any orientation-preserving parametrization $\gamma:[0,2\pi]\rightarrow \Gamma$, if we write it in complex notation as $\gamma(t)=\gamma_1(t)+i\gamma_2(t),$ it holds that 
  $$\operatorname{wind}(\gamma,0)=\frac{1}{2\pi i}\oint_\gamma\frac{\dd z}{z}.$$
  \item
  If $\Gamma$ is a regular $C^2$ Jordan curve oriented
  counterclockwise and $N$ is its exterior unit normal, then
  $\deg_\Gamma(N)=1$.
\end{enumerate}
\end{lemma}

\subsection{Organization of the paper}\label{subsec:organization}

In Section~\ref{sec:auxiliary} we prove implication \ref{itemii} $\Rightarrow$ \ref{itemi}   of Theorem \ref{thm:main}, i.e. we show that
an irrotational flow satisfying the boundary condition and $v_r>0$ at
infinity has no stagnation points. In Section~\ref{sec:proof-main} we
show implication  \ref{itemi} $\Rightarrow$ \ref{itemii} of Theorem~\ref{thm:main} by comparing
the area of a tube of trajectories with the circulation along its
boundary.

\section{Irrotational flows have no stagnation points}
\label{sec:auxiliary}

We prove implication \ref{itemii} $\Rightarrow$ \ref{itemi}
of Theorem~\ref{thm:main}.  We first state the argument principle in the form used below.

\begin{lemma}\label{lem:annular-argument-principle}
Let $\obs\subset\R^2$ be a compact domain, with $0\in\obs$, whose
boundary is a regular $C^2$ Jordan curve.
Assume that $\obs\subset B_R$ and set
$A_R:=B_R\setminus\obs$. Let
$f\in C(\overline{A_R};\mathbb C)$ be holomorphic in $A_R$ and
nonzero on $\partial A_R$.

Let $c_R(\theta):=Re^{i\theta}$ for $0\leq\theta\leq2\pi$, and let $\gamma:[0,2\pi]\to\partial\obs$ be a regular counterclockwise
$C^2$ parametrization, with $\gamma(0)=\gamma(2\pi)$ and $\gamma$
injective on $[0,2\pi)$. Then $f$ has finitely many zeros in $A_R$ and

\begin{equation}\label{eq:annular-argument-principle}
  \sum_{\substack{z\in A_R\\ f(z)=0}}
  \operatorname{ord}_z(f)
  =
  \operatorname{wind}(f\circ c_R,0)
  -
  \operatorname{wind}(f\circ\gamma,0).
\end{equation}
\end{lemma}

\begin{proof}
We apply the argument principle, see
\cite[Chapter~4, Section~5.2]{Ahlfors1979}. Since $f$ is only
assumed holomorphic in $A_R$, we first replace its boundary curves
by nearby curves contained in $A_R$.

Set $m:=\min_{\partial A_R}|f|>0$. By continuity and compactness,
there exists $\delta>0$ such that
\[
  |f(z)|\geq \frac m2
  \qquad
  \textup{if }z\in\overline{A_R}
  \textup{ and }\operatorname{dist}(z,\partial A_R)\leq\delta.
\]
Thus all zeros of $f$ lie in a compact subset of $A_R$.
Since $f$ is not identically zero, the identity theorem implies
that there are only finitely many of them.

Let $N$ be the exterior unit normal to $\partial\obs$ and define
\[
  \gamma_\varepsilon(t)
  :=
  \gamma(t)+\varepsilon N(\gamma(t)),
  \qquad
  c_{R,\varepsilon}(\theta)
  :=
  (R-\varepsilon)e^{i\theta}.
\]
The $C^2$ regularity of $\partial\obs$ ensures that, for sufficiently
small $\varepsilon>0$, these curves bound an annular domain
$A_{R,\varepsilon}\Subset A_R$. We choose $\varepsilon$ small enough
that this domain contains every zero of $f$ and that the deformations
of both boundary curves remain in the region where $f$ does not
vanish. Homotopy invariance then gives
\[
  \operatorname{wind}(f\circ c_{R,\varepsilon},0)
  =
  \operatorname{wind}(f\circ c_R,0),
  \qquad
  \operatorname{wind}(f\circ\gamma_\varepsilon,0)
  =
  \operatorname{wind}(f\circ\gamma,0).
\]

The function $f$ is holomorphic in a neighborhood of
$\overline{A_{R,\varepsilon}}$. We can therefore apply the usual
argument principle on $A_{R,\varepsilon}$. Its outer boundary is
oriented counterclockwise and its inner boundary clockwise, whereas
both parametrizations above are counterclockwise. Hence
\[
  \sum_{\substack{z\in A_{R,\varepsilon}\\ f(z)=0}}
  \operatorname{ord}_z(f)
  =
  \operatorname{wind}(f\circ c_{R,\varepsilon},0)
  -\operatorname{wind}(f\circ\gamma_\varepsilon,0).
\]
Since $A_{R,\varepsilon}$ contains every zero of $f$ in $A_R$,
the preceding identities give
\eqref{eq:annular-argument-principle}.
\end{proof}
The following proposition proves implication \ref{itemii} $\Rightarrow$ \ref{itemi} of Theorem~\ref{thm:main}.

\begin{proposition}\label{prop:irrot-no-stag}
Let $\obs\subset\R^2$ be a compact domain, with $0\in\obs$, whose
boundary is a regular $C^2$ Jordan curve.
Define $\Om:=\R^2\setminus\obs$.
Suppose that
$v\in C^1(\overline\Om;\R^2)$ satisfies
\begin{equation}\label{eq:irrotational-hypotheses}
  \diver v=0,
  \qquad
  \omega:=\partial_1v_2-\partial_2v_1=0
  \qquad\textup{in }\Om,
\end{equation}
and $v\cdot N>0$ on $\partial\obs$, where $N$ is the unit normal
pointing from $\obs$ into $\Om$. Assume that there exists $R_0>0$
such that $\obs\subset B_{R_0}$ and
\begin{equation}\label{eq:outward-sign-argument}
  v(x)\cdot\er(x)>0
  \qquad\textup{for }|x|\geq R_0,
\end{equation}
where $\er(x):=x/|x|$. Then $v(x)\neq0$ for every
$x\in\overline\Om$.
\end{proposition}

\begin{proof}
Set
$q(z):=v_1(z)-iv_2(z)$. The identity
\begin{equation}\label{eq:dbar-q-local}
  2\partial_{\overline z}q
  =
  \diver v-i\omega
\end{equation}
shows that $q$ is holomorphic in $\Om$. It is also continuous on
$\overline\Om$, and
\begin{equation}\label{eq:q-v-zero}
  q(z)=0\quad\Longleftrightarrow\quad v(z)=0.
\end{equation}

Let $\gamma:[0,2\pi]\to\partial\obs$ be a regular counterclockwise
$C^2$ Jordan parametrization. Write $\mathsf v:=v_1+iv_2$, so that
$q=\overline{\mathsf v}$. On $\partial\obs$, consider $
  v_s:=(1-s)v+sN,
\, 0\leq s\leq1.$
$v_s$ is never zero,  since
\[
  v_s\cdot N=(1-s)v\cdot N+s>0.
\]
By Lemma~\ref{lem:winding-properties},
\[
  \operatorname{wind}(\mathsf v\circ\gamma,0)
  =
  \deg_{\partial\obs}(N)=1.
\]
Taking complex conjugates therefore gives
\begin{equation}\label{eq:inner-winding}
  \operatorname{wind}(q\circ\gamma,0)=-1.
\end{equation}

Fix $R\geq R_0$ and parametrize $\partial B_R$ counterclockwise by
$c_R(\theta):=Re^{i\theta}$, where $0\leq\theta\leq2\pi$.
On this circle, consider the homotopy
$
  w_s:=(1-s)v+s\er,
  \, 0\leq s\leq1.
$
Condition~\eqref{eq:outward-sign-argument} gives
\[
  w_s\cdot\er=(1-s)v\cdot\er+s>0,
\]
so this homotopy does not vanish either. The complex representation
of $\er(c_R(\theta))$ is $e^{i\theta}$, whose winding number is $1$.
Thus $\operatorname{wind}(\mathsf v\circ c_R,0)=1,$
and taking complex conjugates yields
\begin{equation}\label{eq:outer-winding}
  \operatorname{wind}(q\circ c_R,0)=-1.
\end{equation}

Set $A_R:=B_R\setminus\obs$. The assumptions on $v$
ensure that $q$ does not vanish on $\partial A_R$.
Applying Lemma~\ref{lem:annular-argument-principle} and using
\eqref{eq:inner-winding} and \eqref{eq:outer-winding}, we obtain
\[
  \sum_{\substack{z\in A_R\\ q(z)=0}}
  \operatorname{ord}_z(q)
  =
  -1-(-1)=0.
\]
Since every zero has positive multiplicity, $q$ has no zeros in
$A_R$. As $R\geq R_0$ is arbitrary, \eqref{eq:q-v-zero} shows
that $v$ does not vanish in $\Om$. Finally, $v\cdot N>0$ excludes
zeros on $\partial\obs$, which completes the proof.
\end{proof}

\section{Flows without stagnation points are irrotational}
\label{sec:proof-main}

We now prove implication \ref{itemi} $\Rightarrow$ \ref{itemii}
of Theorem~\ref{thm:main}. We first study the trajectories and
show that they tend to infinity when the velocity has no zeros.
We then prove that the vorticity is constant along trajectories.
Finally, we construct a tube of trajectories and compare two
estimates for the integral of the vorticity over this tube.

\subsection{Geometry of the streamlines}\label{subsec:dynamics}
Since $v\in C^1(\overline\Om;\mathbb R^2)$, its restriction to $\Om$ is
locally Lipschitz. Therefore, for every $x\in\Om$, there exists a unique
maximal solution
\[
  X(\,\cdot\,,x):(T_-(x),T_+(x))\longrightarrow\Om
\]
of
\begin{equation}\label{eq:flow-ode}
  \frac{\dd}{\dd t}X(t,x)=v(X(t,x)),
  \quad
  X(0,x)=x,
\end{equation}
where $-\infty\leq T_-(x)<0<T_+(x)\leq+\infty$. Standard ODE theory \cite{Perko2001} shows that the set 
\begin{equation}\label{eq:domainflow} 
V=\{(t,x)\in \R\times \R^2\,: x\in \Omega,\, t\in (T_-(x),T_+(x))\}
\end{equation}
is open and that the flow $(t,x)\mapsto X(t,x)$ is $C^1$ on $V$.
For each fixed $x\in\Om$, the trajectory $t\mapsto X(t,x)$ is $C^2$
on $(T_-(x),T_+(x))$. Since \eqref{eq:flow-ode} is autonomous,
uniqueness gives $X(t+s,x)=X(t,X(s,x))$ whenever both sides are defined.
The next lemma proves that $T_+(x)=\infty$ for every $x\in\Om$.

\begin{lemma}
\label{lem:boundary-repulsion}
{
Assume that $v\in C^1(\overline\Om;\R^2)$ satisfies
\eqref{eq:source-boundary} and that there exist
$R_\infty,L_\infty>0$ such that
\[
  |v(x)|\leq \frac{L_\infty}{|x|}
  \qquad\textup{for }|x|\geq R_\infty.
\]
} 
Then, for every $x\in\Om$, the trajectory $X(t,x)$ is defined for every
$t\geq0$ and remains in $\Om$. {Furthermore,} 
there exists $\varepsilon>0$ such that the open set
\begin{equation}\label{eq:boundary-collar}
  U_\varepsilon
  :=
  \left\{
    y\in\Om:0<\dist(y,\obs)<\varepsilon
  \right\}
\end{equation}
has the following properties:
\begin{enumerate}[label=\textup{(\roman*)}]
  \item \label{item:firstitemlemmaboundary-repulsion} If $x\in U_\varepsilon$, then the trajectory exits
  $U_\varepsilon$ in finite positive time.
  \item \label{item:seconditemlemmaboundary-repulsion} After leaving $U_\varepsilon$, the forward trajectory can never re-enter it.
  \item \label{item:thirditemlemmaboundary-repulsion} If $x\notin U_\varepsilon$, then the forward trajectory never enters
  $U_\varepsilon$.
\end{enumerate}
\end{lemma}

\begin{proof}

Let $d$ be the signed distance to $\partial\obs$, chosen positive in $\Om$. By the tubular neighborhood theorem (see \cite[Section~2.2]{AbateTovena2012}), there exists $\delta_0>0$ such that $d$ is $C^2$ on the open set $\{x\in \Omega\,:\,d(x)<\delta_0\}$, and $\nabla d=N$ on $\partial\obs$. 
Compactness of $\partial\obs$ and  \eqref{eq:source-boundary} yield  $\mu:=\min_{\partial\obs}(v\cdot N)>0. $ 
Since $y\mapsto v\cdot\nabla d$ is continuous in $\{0\leq d<\delta_0\}$, we obtain
\begin{equation}\label{eq:distance-along-flow}
  v(y)\cdot\nabla d(y)\geq c_0:=\frac\mu2
  \quad\textup{whenever }0\leq d(y)\leq\delta,
\end{equation}
for some $0<\delta<\delta_0$.
Along any trajectory in $U_\delta$, we have
\begin{equation}\label{eq:monotonicity}
  \frac{\dd}{\dd t}d(X(t,x))
  =v(X(t,x))\cdot\nabla d(X(t,x))\geq c_0.
\end{equation} \medskip
Define $\varepsilon:=\delta/2.$ We now show that $T_+(x)=\infty$. We argue by contradiction. Assume that $T_+(x)<\infty.$ 
Since $v$ is continuous, and due to \eqref{eq:speed-upper},
$v\in L^\infty(\overline\Om)$. For $0\leq s<t<T_+(x)$, we have
$$|X(t,x)-X(s,x)|\leq \int_{s}^t|v(X(\sigma,x))|\dd\sigma \leq \|v\|_{L^\infty(\overline{\Om})}|t-s|.$$
Therefore, $t\mapsto X(t,x)$ is Lipschitz and extends continuously to $[0,T_+(x)]$.
By \eqref{eq:monotonicity}, the distance $d(X(t,x))$ is strictly
increasing whenever $X(t,x)\in U_\delta$. Consequently,
\[
  d(X(t,x))\geq\min\{d(x),\varepsilon\}>0
  \qquad\textup{for }0\leq t<T_+(x).
\]

Since $\obs$ is closed, the limit satisfies $X(T_+(x),x)\in\Om$. Then, there exists $\alpha>0$ such that the following ODE
$$\left\{\begin{array}{l}
c'(t)=v(c(t)),\\
c(0)=X(T_+(x),x)\in \Omega,
\end{array}\right.$$
admits a solution on $t\in(-\alpha,\alpha)$. Therefore, the function
\[
  F(t):=
  \begin{cases}
    X(t,x), & 0\leq t<T_+(x),\\
    c(t-T_+(x)), & T_+(x)\leq t<T_+(x)+\alpha.
  \end{cases}
\]
is a solution of \eqref{eq:flow-ode} that extends $X(t,x)$ beyond $T_+(x)$. This contradicts the maximality of the solution $X(\cdot,x)$. \medskip

We now prove
\ref{item:firstitemlemmaboundary-repulsion}. We argue by contradiction.
If $X(t,x)\in U_\varepsilon$ for every positive time, then
\eqref{eq:monotonicity} implies that
{
\[
  d(X(t,x))\geq d(x)+c_0t.
\]
Consequently, for
\[
  t\geq\frac{\varepsilon-d(x)}{c_0},
\]
we obtain $d(X(t,x))\geq\varepsilon$, which contradicts
$X(t,x)\in U_\varepsilon$. This proves
\ref{item:firstitemlemmaboundary-repulsion}.

We now prove
\ref{item:seconditemlemmaboundary-repulsion}. Let
$x\in U_\varepsilon$, and consider}
{
\[
  t_0
  :=
  \inf\left\{
    t\geq0:
    d(X(t,x))\geq\varepsilon
  \right\}.
\]
By \ref{item:firstitemlemmaboundary-repulsion}, $t_0<\infty$, and
continuity gives
\[
  d(X(t_0,x))=\varepsilon.
\]
Since $\varepsilon=\delta/2$, continuity and
\eqref{eq:monotonicity} imply that there exists $\alpha>0$ such that
\[
  \varepsilon<d(X(t,x))<\delta
  \qquad\textup{for }t\in(t_0,t_0+\alpha).
\]
Define the first possible return time by
\[
  \tau
  :=
  \inf\left\{
    t>t_0+\frac{\alpha}{2}:
    d(X(t,x))\leq\varepsilon
  \right\}.
\]
The preceding inequality implies that
$\tau>t_0+\alpha/2$. Suppose that $\tau<\infty$. By continuity of
$X(t,x)$ and $d$,
\[
  d(X(\tau,x))=\varepsilon.
\]
Moreover, there exists $\beta>0$ such that
\[
  d(X(t,x))\in(\varepsilon,\delta)
  \qquad\textup{for every }t\in(\tau-\beta,\tau).
\]
However, this is impossible, because
\eqref{eq:monotonicity} implies that
$t\mapsto d(X(t,x))$ is strictly increasing on
$(\tau-\beta,\tau]$. Therefore $\tau$ cannot be finite. We conclude that
\[
  d(X(t,x))>\varepsilon
  \qquad\textup{for every }t>t_0.
\]
This proves
\ref{item:seconditemlemmaboundary-repulsion}. \medskip

Finally, suppose that $x\notin U_\varepsilon$. Then
$d(x)\geq\varepsilon$. If $d(x)=\varepsilon$,
\eqref{eq:monotonicity} implies that $d(X(t,x))>\varepsilon$ for all
sufficiently small $t>0$. If the trajectory entered $U_\varepsilon$ at a
later time, there would be a first time $\tau>0$ at which it crossed the
level $d=\varepsilon$ from above. For $t<\tau$ sufficiently close to
$\tau$, one would have
$\varepsilon<d(X(t,x))<\delta$, whereas
\eqref{eq:monotonicity} would make $t\mapsto d(X(t,x))$ strictly
increasing there. This is impossible. Thus the trajectory never enters
$U_\varepsilon$, proving
\ref{item:thirditemlemmaboundary-repulsion}.}
\end{proof}

\begin{remark}
The bound $|v(x)|\leq L_\infty/|x|$ outside a large disk is used in this
lemma only to deduce that $v$ is globally bounded. More generally, forward completeness follows under
any growth condition that prevents trajectories from reaching infinity in
finite time, such as $|v(x)|\leq C(1+|x|).$
\end{remark}

We use the following elementary consequence of the degree.

\begin{lemma}
\label{lem:index}
Let $D\subset\mathbb R^2$ be a bounded Jordan domain whose boundary
$\partial D$ is a regular $C^2$ Jordan curve. Let $w\in C^0(\overline D;\mathbb R^2).$
Assume that $w$ is tangent and nonzero on $\partial D$. Then there exists
$x_*\in D$ such that $w(x_*)=0.$
\end{lemma}

\begin{proof}
Suppose, toward a contradiction, that $w$ has no zero in $\overline D$.
Then
\[
  F:=\frac{w}{|w|}:\overline D\longrightarrow\mathbb S^1
\]
is continuous. Since a Jordan domain is homeomorphic to a disk, it is contractible. Therefore, the
restriction restriction $F|_{\partial D}$ is homotopic to a constant map, and $\deg_{\partial D}(F|_{\partial D})=0$.

Choose the positive orientation on $\partial D$ and let $T$ be its unit
tangent. Tangency and nonvanishing imply
\[
  F|_{\partial D}=\sigma T,
\]
where the continuous function $\sigma:\partial D\to\{-1,1\}$ is constant
because $\partial D$ is connected. The turning-tangent theorem for regular
$C^1$ Jordan curves gives
\[
  \deg_{\partial D}(T)=1,
\]
see~\cite[Section~2.4]{AbateTovena2012}. Multiplication by $-1$ is a rotation
of $\mathbb S^1$ and does not change the degree. Hence
\[
  \deg(F|_{\partial D})=1,
\]
contradicting the extension of $F$ to $\overline D$. Therefore $w$ has a
zero in $D$.
\end{proof}

We next exclude periodic streamlines. We call $\Gamma\subset\Om$ a periodic
orbit of \eqref{eq:flow-ode} if there exist $x_0\in\Om$ and $T>0$ such that
$X(T,x_0)=x_0$ and the trajectory $t\mapsto X(t,x_0)$ is not constant. Its
image is then
\[
  \Gamma:=\{X(t,x_0):0\leq t\leq T\}.
\]

\begin{lemma}\label{lem:no-periodic}
Assume that $v\in C^1(\overline\Om;\mathbb R^2)$ is divergence free in
$\Om$, is nonzero on $\overline\Om$, and satisfies the strict source
condition \eqref{eq:source-boundary}. Then the flow generated by $v$ has no
periodic orbits in $\Om$.
\end{lemma}

\begin{proof}
We argue by contradiction. Suppose that there exists a nonconstant periodic
trajectory through some $x_0\in\Om$. Choose its least positive period $T$.
Since $v$ never vanishes,
standard uniqueness for autonomous ODEs shows that
\[
  \gamma(t):=X(t,x_0),\quad 0\leq t< {\infty},
\]
is injective on $[0,T)$, {periodic with period $T$} and satisfies $\gamma'(t)=v(\gamma(t))\neq0$.
Moreover, $v\in C^1$ implies $\gamma\in C^2$. Hence
$\Gamma:=\gamma([0,T])$ is a regular $C^2$ Jordan curve.

By the Jordan curve theorem, $\R^2\setminus\Gamma$ can be written as the union of two connected components, $\mathbb R^2\setminus\Gamma = D_\Gamma\cup E_\Gamma,$ where $D_\Gamma$ is bounded and $E_\Gamma$ is the unbounded component.

Since $\Gamma\subset\Om$, we have $\Gamma\cap\obs=\varnothing.$
The set $\obs$ is connected, and therefore it must lie entirely in
one of the two connected components of $\mathbb R^2\setminus\Gamma$. Thus
exactly one of the following alternatives holds:
\begin{enumerate}[label=\textup{(\roman*)}]
  \item $\obs\subset E_\Gamma$.
  \item $\obs\subset D_\Gamma$.
\end{enumerate}

\medskip
\noindent
\underline{\emph{Case 1: $\obs\subset E_\Gamma$.}}
In this case, $\overline{D_\Gamma}\cap\obs=\varnothing,$
and hence $\overline{D_\Gamma}\subset\Om.$
The vector field $v$ is continuous on $\overline{D_\Gamma}$, is nonzero
there, and is tangent and nonzero on $\Gamma=\partial D_\Gamma$, because $v(\gamma(t))=\gamma'(t).$
Applying Lemma \ref{lem:index} to $D_\Gamma$ and
$w=v|_{\overline{D_\Gamma}}$, we obtain a point
$x_*\in D_\Gamma$ such that $v(x_*)=0.$
This contradicts the assumption that $v$ is nonzero in
$\overline\Om$.

\medskip
\noindent
\underline{\emph{Case 2: $\obs\subset D_\Gamma$.}}
Consider the bounded annular domain
$A_\Gamma:=D_\Gamma\setminus\obs.$
 Its boundary has two components, $\Gamma$ and $\partial\obs$.  
Let $n_\Gamma$ denote the outward unit normal to
$D_\Gamma$ along $\Gamma$. The outward unit normal to $A_\Gamma$ is
$n_\Gamma$ on $\Gamma$ and $-N$ on $\partial\obs$.

The divergence theorem and $\diver v=0$ give
\begin{align}
  0=
  \int_{A_\Gamma}\diver v \dd x =
  \int_\Gamma v\cdot n_\Gamma \dd s 
  -
  \int_{\partial\obs}v\cdot N \dd s .
  \label{eq:periodic-orbit-flux}
\end{align}
Since $v$ is tangent to $\Gamma$, the first boundary integral vanishes,
whereas 
\eqref{eq:source-boundary} makes the second strictly positive. This
contradicts \eqref{eq:periodic-orbit-flux}.

Both cases, 1 and 2, lead to a contradiction. Thus the
flow has no periodic orbits in $\Om$.
\end{proof}

\begin{remark}
The two cases use two different types of arguments. A contractible periodic orbit is
excluded by the planar index of a tangent vector field. A periodic orbit
surrounding the obstacle is excluded by incompressibility and the  positive
flux through the obstacle boundary.
\end{remark}

\begin{lemma}
\label{lem:escape}
Assume $v\in C^1(\overline\Om;\mathbb R^2)$, $\diver v=0$ in $\Om$,
\eqref{eq:source-boundary} and \eqref{eq:far-field-cost}. Suppose that
$v(x)\neq0$ for every $x\in\overline\Om$. Then, for every $x\in\Om$,
\begin{equation}\label{eq:escape}
  \lim_{t\to\infty}|X(t,x)|=\infty.
\end{equation}
Moreover, there exists $\overline R>0$ such that, for every $x\in\Om$, there is a time $\overline t=\overline t(x)\geq0$ for which $t\mapsto|X(t,x)|$ is strictly increasing on $[\overline t,\infty)$ and $v_r(X(t,x))>0$ for $t\geq\overline t$.
For every $R>\max\{\overline R,|x|\}$, this part of the trajectory crosses $\partial B_R$ exactly once.

\end{lemma}

\begin{proof}
Fix $x\in\Om$ and write $X(t):=X(t,x)$. By
\eqref{eq:speed-upper} and Lemma \ref{lem:boundary-repulsion}, the trajectory is
defined for every $t\geq0$ and remains in $\Om$.

Suppose that the forward orbit were bounded. Lemma
\ref{lem:boundary-repulsion} then places a positive tail of the orbit in a
compact subset of $\Om$. Its $\omega$-limit set is nonempty, compact and contains
no equilibrium. The Poincar\'e--Bendixson theorem implies that its $\omega$-limit set is a periodic orbit. See Perko~\cite[Section~3.7, Theorem~1]{Perko2001}. This contradicts Lemma
\ref{lem:no-periodic}. Hence
\begin{equation}\label{eq:unbounded-forward-orbit}
  \sup_{t\geq0}|X(t,x)|=\infty.
\end{equation}
Set $\overline R:=R_0$, where $R_0$ is as in \eqref{eq:far-field-cost}, and $\widetilde R:=\max\{\overline R,|x|\}$. Define
\[
  {\overline{t}}:=\inf\{t\geq0:|X(t)|={\widetilde{R}}\}.
\]
This time is finite by \eqref{eq:unbounded-forward-orbit}. Whenever
$|X(t)|\geq \overline{R}$, condition \eqref{eq:far-field-cost} gives
\begin{equation}\label{eq:positive-radial-derivative}
  \frac{\dd}{\dd t}|X(t)|
  =v(X(t))\cdot\er(X(t))
  =v_r(X(t))>0.
\end{equation}
The trajectory cannot cross $\partial B_{\widetilde R}$ inwards,
because its radial derivative is positive there. Hence
$|X(t)|\geq\widetilde R$ for $t\geq\overline t$, and
$t\mapsto|X(t)|$ is strictly increasing on this interval.
By \eqref{eq:unbounded-forward-orbit}, it follows that
$|X(t)|\to\infty$ as $t\to\infty$. Therefore, since $|X(\overline{t})|=\widetilde{R}$, by continuity, for every $R>\widetilde{R}$, there exists $t>\overline{t}$ such that $|X(t)|=R$, so the curve $(\overline{t},\infty)\ni t\mapsto X(t,x)$ crosses every circle $\{|x|=R\}$ for $R>\widetilde{R}$. Furthermore, as $|X(t)|$ is strictly increasing, it does so only once. Finally, the condition on $\partial_t X(t,x)\cdot e_r(X(t,x))$ follows immediately from \eqref{eq:positive-radial-derivative}. 
\end{proof}

\subsection{Transport of vorticity}\label{subsec:vorticity-transport}

We next prove that the vorticity is constant along trajectories.
Since $v$ is only assumed to be $C^1$, the vorticity is continuous
and the transport equation must first be understood in the
sense of distributions.

\begin{lemma}\label{lem:vorticity-transport}
Assume  \eqref{eq:euler-intro} and \eqref{eq:regularity} and that $v(x)\neq0$ for every $x\in\Om$. Then, for every
$x\in\Om$,
\begin{equation}\label{eq:omega-flow}
  \omega(X(t,x))=\omega(x),
\end{equation}
for every time $t$ for which the trajectory $X(t,x)$ is defined.
\end{lemma}
\begin{proof}
Since $v\in C^1(\overline\Om;\mathbb R^2)$, we have
$\omega\in C^0(\overline\Om)$. The Euler equations give
\[
  -\nabla p=(v\cdot\nabla)v
  =
  \nabla\left(\frac12|v|^2\right)+\omega v^\perp.
\]
Using this identity and integrating by parts, we obtain
\begin{equation}\label{eq:distributional-vorticity-transport}
\begin{split}
  \int_{\Omega}(v\cdot\nabla\varphi)\omega
   \dd x 
  &=
  \int_{\Omega}
  (v^\perp\cdot\nabla^\perp\varphi)\omega
   \dd x \\
  &=
  -\int_{\Omega}(\nabla^\perp\varphi)\cdot
  \nabla\left(p+\frac12|v|^2\right)
   \dd x \\
  &=
  \int_{\Omega}
  \diver(\nabla^\perp\varphi)
  \left(p+\frac12|v|^2\right)
   \dd x 
  =0,
  \quad
  \textup{for every }\varphi\in C_c^2(\Omega).
\end{split}
\end{equation}

A standard approximation argument allows us to
extend \eqref{eq:distributional-vorticity-transport} to any
$\varphi\in C_c^1(\Omega)$. Therefore, $\omega$ satisfies the
transport equation
\[
  (v\cdot\nabla)\omega=0
\]
in a distributional sense.
We now use local flow coordinates to prove that
$\omega$ is constant along the field lines of $v$.
If $v(x)=0$, uniqueness for the ODE gives $X(t,x)=x$, and
the conclusion is straightforward. We now consider a point
$x_0\in\Om$ such that $v(x_0)\neq0$. Define the curve $\sigma(s)$ by

\[
  \left\{
  \begin{array}{l}
    \sigma'(s)=v^\perp(\sigma(s)),\\
    \sigma(0)=x_0.
  \end{array}
  \right.
\]

Since $v\in C^1(\Omega;\R^2)$, for every $x_0\in\Omega$
there exists a $\delta>0$ such that $\sigma(s)$ is defined on
$(-\delta,\delta)$. We recall that the set $V$ in \eqref{eq:domainflow} is an open
set, and that the function
$\R\times(-\delta,\delta)\ni(t,s)\mapsto(t,\sigma(s))$
is continuous. Therefore, the set
\[
  W_{x_0}
  :=
  \{(t,s)\in\R\times\R:
    s\in(-\delta,\delta),\ (t,\sigma(s))\in V\}
\]
is open, so there exists a rectangle
$Q_{x_0}:=(-a,a)\times(-a,a)\subset W_{x_0}$.
We now define the function
\[
  F(t,s):=X(t,\sigma(s)).
\]
The map $F:Q_{x_0}\to F(Q_{x_0})$ is in
$C^1(Q_{x_0};\R^2)$ and its Jacobian matrix at $(0,0)$ is given by
\[
 DF(0,0)= \begin{pmatrix}
    v_1(x_0)&-v_2(x_0)\\
    v_2(x_0)&v_1(x_0)
  \end{pmatrix}.
\]

Thus, its Jacobian determinant, $J$, equals
$J(0,0)=|v(x_0)|^2>0$. Therefore, by taking $Q_{x_0}$
smaller, the map $F:Q_{x_0}\longrightarrow F(Q_{x_0})$
is a $C^1$ diffeomorphism.

Furthermore, $\partial_tF(t,s)=v(F(t,s))$. Therefore,
$\partial_tF(t,s)$ is also in $C^1(Q_{x_0};\R^2)$.
Thus, its derivatives satisfy
\[
  \frac{\partial}{\partial t}\frac{\partial F}{\partial t}
  =
  \nabla v(F(t,s))\cdot\frac{\partial F}{\partial t}
  \quad\textup{and}\quad
  \frac{\partial}{\partial t}\frac{\partial F}{\partial s}
  =
  \nabla v(F(t,s))\cdot\frac{\partial F}{\partial s}.
\]
 
Liouville's formula (see \cite[Section~3.5]{Perko2001}) gives
\[
  \partial_tJ
  =
  \operatorname{tr}\big(\nabla v(F(t,s))\big)J
  =
  J\big((\diver v)\circ F\big)=0.
\]
Thus $J(t,s)=J(0,s)$.

Let $\varphi\in C_c^\infty(Q_{x_0})$ and set
$\psi:=\varphi\circ F^{-1}$. The function $\psi$, extended by
zero, is an admissible $C_c^1(\Om)$ test function in
\eqref{eq:distributional-vorticity-transport}.
Since $\partial_tF=v\circ F$, changing variables gives
\[
  0
  =
  \int_{F(Q_{x_0})}
  \omega(x)\big(v\cdot\nabla\psi(x)\big)
   \dd x 
  =
  \int_{Q_{x_0}}
  \omega(F(t,s))J(0,s)\partial_t\varphi(t,s)
  {\dd t}{\dd s}
\]
for any $\varphi\in C_c^\infty(Q_{x_0})$.
Therefore,
\[
  \partial_t\bigl(J\,\omega\circ F\bigr)=0
  \quad\textup{in }\mathcal D'(Q_{x_0}).
\]
Since $J$ is positive and independent of $t$, it follows that
$\omega\circ F$ is independent of $t$ almost everywhere.
The continuity of $\omega\circ F$ upgrades this identity to
every $(t,s)\in Q_{x_0}$. Therefore, $\omega(X(t,x_0))=\omega(x_0)$ whenever
$(t,0)\in Q_{x_0}$. \medskip

We now show that
$\omega(X(t,x))=\omega(x)$ for every $t>0$.
To do so, consider the time
\[
  t_\star
  :=
  \sup\{t>0:
    \omega(X(s,x))=\omega(x)
    \textup{ for every }s\leq t\}.
\]

We just saw that $t_\star>0$. We then show that
$t_\star=\infty$. If it was not the case, by continuity of
$\omega$ and $X(t,x)$,
$\omega(X(t_\star,x))=\omega(x)$.
Finally, we can repeat the argument above taking
$x_0=X(t_\star,x)$, so there exists some $\alpha>0$ such that
$\omega(X(\sigma,x_0))=\omega(x_0)$ for any
$\sigma\in(0,\alpha)$. Since
\[
  X(\sigma,x_0)
  =
  X(\sigma,X(t_\star,x))
  =
  X(\sigma+t_\star,x),
\]
and $\omega(x_0)=\omega(x)$, we conclude that $t_\star$
cannot be maximal, and this is a contradiction.

\end{proof}

\subsection{Completion of the proof}\label{subsec:proof-main}

We now prove implication \ref{itemi} $\Rightarrow$ \ref{itemii}
of Theorem~\ref{thm:main}, using the results of the preceding
two subsections. We divide the proof into several steps.\medskip

Assume that $(v,p)$ satisfies the hypotheses of
Theorem~\ref{thm:main} and that $v(x)\neq0$ for every
$x\in\overline\Om$. Suppose, toward a contradiction, that
$\omega(x_*)\neq0$ at some point $x_*\in\Om$.

\medskip
\noindent
\underline{\emph{(a) An initial arc where the vorticity does not vanish.}}
By Lemma \ref{lem:escape}, the trajectory through $x_*$ escapes to infinity.
Choose
\[
  R_1>\max\left\{{\overline{R}},|x_*|,
  \sup_{z\in\obs}|z|\right\},
\]
where ${\overline{R}}$ is the one in Lemma \ref{lem:escape}. Define the first hitting time
\[
  t_*:=\inf\{t\geq0:|X(t,x_*)|=R_1\}.
\]
By Lemma~\ref{lem:escape}, this time is finite, the crossing is transverse and
outward, and the trajectory does not return to $\partial B_{R_1}$.
Set $y_*:=X(t_*,x_*)$. By Lemma \ref{lem:vorticity-transport},
$\omega(y_*)=\omega(x_*)\neq0$.

Let $\sigma:=\sign(\omega(y_*))$. By continuity of $\omega$, there exist
$\delta>0$ and a nondegenerate closed proper arc
$I_{R_1}\subset\partial B_{R_1}$ containing $y_*$ in its relative interior
such that
\begin{equation}\label{eq:omega-on-initial-arc}
  \sigma\omega\geq\delta
  \ \textup{on }I_{R_1}.
\end{equation}
Let $\gamma_1:[s_-,s_+]\to I_{R_1}$ be an injective regular
$C^1$ parametrization chosen as the restriction of a regular
parametrization of $\partial B_{R_1}$ on a slightly larger
open interval. Since $v_r>0$ on $\partial B_{R_1}$, we have that
\[
  \det\bigl(v(\gamma_1(s)),\gamma_1'(s)\bigr)\neq0,
  \qquad\textup{for }s\in[s_-,s_+].
\]

\medskip
\noindent
\underline{\emph{(b) Transverse sections of the tube.}}

For $s\in[s_-,s_+]$ and $\rho\geq R_1$, define
\[
  \tau(s,\rho)
  :=
  \inf\{t\geq0:|X(t,\gamma_1(s))|=\rho\}.
\]

{{By Lemma} \ref{lem:escape},}
$\tau(s,\rho)$ is finite and unique. We now justify the regularity of the hitting time. Use the chosen extension of
$\gamma_1$ to a slightly larger open interval and define
\[
  G(t,s,\rho):=|X(t,\gamma_1(s))|-\rho \quad {\textup{for }t\geq 0, s\in [s_-,s_+] \textup{ and }\rho\geq R_1}.
\]
{{By Lemma} \ref{lem:escape}, $G(t,s,\rho)=0$ {if and only if} } $t=\tau(s,\rho)$. {At such points, it holds that }
\begin{equation}\label{eq:hitting-time-transversality}
  \partial_tG(\tau(s,\rho),s,\rho)
  =v_r\bigl(X(\tau(s,\rho),\gamma_1(s))\bigr)>0.
\end{equation}
{{By the implicit function theorem}, there exists a unique $C^1$ map $\varphi(s,\rho)$  such that $|X(\varphi(s,\rho),\gamma_1(s))|-\rho=0$. {By uniqueness of} $\tau(s,\rho)$, we conclude that $\varphi(s,\rho)=\tau(s,\rho)$, and thus  $(s,\rho)\mapsto \tau(s,\rho)$ is $C^1$} 

on every compact rectangle
$\mathcal{R}_R:= [s_-,s_+]\times[R_1,R]$. 
We now define
\begin{equation}\label{eq:stream-tube-parametrization}
  F(s,\rho)
  :=
  X(\tau(s,\rho),\gamma_1(s)).
\end{equation}
We first show that $F$ is injective. To prove it, consider $(s,\rho)$ and $(s',\rho')$ such that
$F(s,\rho)=F(s',\rho')$. Then, by definition of $\tau$,
\[
  \rho
  =|X(\tau(s,\rho),\gamma_1(s))|
  =|X(\tau(s',\rho'),\gamma_1(s'))|
  =\rho'.
\]
On the other hand, {by uniqueness for}
the ODE \eqref{eq:flow-ode}, if
$X(\tau(s,\rho),\gamma_1(s))$ and
$X(\tau(s',\rho),\gamma_1(s'))$ are equal, then
$\gamma_1(s)$ and $\gamma_1(s')$ have to belong to the same
field line of $v$. But this is only possible if
$\gamma_1(s)=\gamma_1(s')$. Indeed,
{if these points were distinct, there would exist}
some $x\in\Omega$ and $t_1,t_2\in[0,\infty)$ such that
$X(t_1,x)=\gamma_1(s)$ and
$X(t_2,x)={\gamma_1(s')}$.
Without loss of generality, take $t_1<t_2$.

Since $|X(t_1,x)|=R_1>R_0$, condition
\eqref{eq:far-field-cost} prevents the trajectory from crossing
$\partial B_{R_1}$ inwards and makes its radial coordinate
strictly increasing for $t\geq t_1$. Hence
\[
  |\gamma_1(s)|=|X(t_1,x)|
  <|X(t_2,x)|=|\gamma_1(s')|.
\]
This is impossible, since both points lie on $\partial B_{R_1}$.
Therefore $\gamma_1(s)=\gamma_1(s')$, and the injectivity of
$\gamma_1$ gives $s=s'$. Together with $\rho=\rho'$, this proves
that $F$ is injective.
 Since $\mathcal R_R$ is compact and $F$ is continuous, $F$
is a homeomorphism from $\mathcal R_R$ onto its image.  \medskip
We also define
\begin{equation}\label{eq:transverse-arc}
  I_\rho:=\{F(s,\rho):s\in[s_-,s_+]\}
  \subset\partial B_\rho.
\end{equation}
We next show that $\partial_sF$ and $\partial_\rho F$ are
linearly independent. Differentiating $|F(s,\rho)|=\rho$
with respect to $\rho$ gives

\[
  \er(F(s,\rho))\cdot\partial_\rho F(s,\rho)=1.
\]

Similarly, differentiating \eqref{eq:stream-tube-parametrization} gives
$\partial_\rho F=v(F)\partial_\rho\tau$. Hence
\[
  1=\er(F)\cdot\partial_\rho F
  =v_r(F)\partial_\rho\tau,
  \quad
  \partial_\rho\tau=\frac1{v_r(F)}.
\]
Consequently,
\begin{equation}\label{eq:rho-direction}
  \partial_\rho F(s,\rho)
  =
  \frac{v(F(s,\rho))}{v_r(F(s,\rho))}.
\end{equation}

The derivative $\partial_sF$ is tangent to $\partial B_\rho$ and does not
vanish. To verify the latter assertion, suppose that
$\partial_sF(s,\rho)=0$. Differentiating
\eqref{eq:stream-tube-parametrization} with respect to $s$ gives
\[
  D_xX(\tau,\gamma_1(s))\gamma_1'(s)
  +v(F(s,\rho))\partial_s\tau(s,\rho)=0.
\]
The differential of the flow is invertible and satisfies
$D_xX(t,z)v(z)=v(X(t,z))$. It would follow that $\gamma_1'(s)$ is parallel
to $v(\gamma_1(s))$, contradicting transversality. Thus $\partial_sF\neq0$.
Since $\partial_sF$ is tangential and
$\er\cdot\partial_\rho F=1$, the two derivatives are linearly
independent. Together with the preceding injectivity argument,
this shows that $F$ is a $C^1$ embedding of $\mathcal R_R$. The inverse function theorem shows that $F$ is a local $C^1$
diffeomorphism in the interior of the parameter rectangle.
The extensions constructed above allow us to apply the same
argument at the boundary. \medskip

\noindent
\underline{\emph{(c) The truncated stream tube.}}
For $R>R_1$, define
\begin{equation}\label{eq:stream-tube}
  T_R
  :=
  F\bigl((s_-,s_+)\times(R_1,R)\bigr).
\end{equation}
Because $F$ is a $C^1$ embedding of the rectangle,
$\overline{T_R}=F(\mathcal R_R)$ and
$\partial T_R=F(\partial\mathcal R_R)$. The horizontal sides give the
transverse arcs $I_{R_1}$ and $I_R$, while the vertical sides give the two
streamline arcs
\[
  \Gamma_-^R:=\{F(s_-,\rho):R_1\leq\rho\leq R\},
  \quad
  \Gamma_+^R:=\{F(s_+,\rho):R_1\leq\rho\leq R\}.
\]
The four sides are regular and meet only at their corresponding corners. Hence $T_R$ is a bounded piecewise $C^1$ Jordan domain, so
the divergence and Green--Stokes theorems apply. By
\eqref{eq:rho-direction}, $v$ is tangent to both $\Gamma_-^R$ and $\Gamma_+^R$, and its 
normal flux {along either of these curves} vanishes. See Figure~\ref{fig:stream-tube-geometry} for an illustration of the geometry described above.

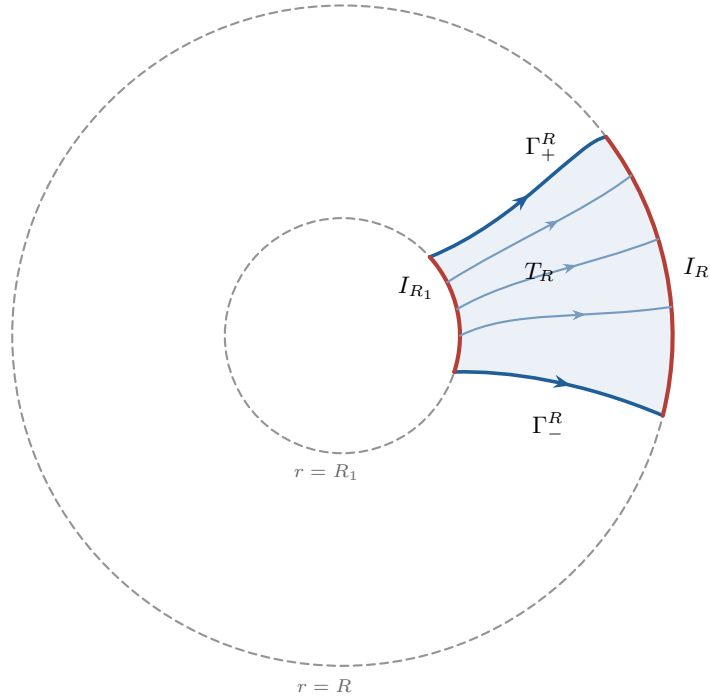
\begin{figure}[h]
  \centering
  \begin{tikzpicture}[
      scale=0.96,
      >=Stealth,
      line cap=round,
      line join=round,
      circleline/.style={
        draw=black!42,
        densely dashed,
        line width=0.8pt
      },
      tubefill/.style={
        fill=outflowblue!9
      },
      side/.style={
        draw=outflowblue,
        line width=1.35pt,
        postaction={decorate},
        decoration={
          markings,
          mark=at position 0.55 with
          {\arrow{Stealth[length=2.3mm]}}
        }
      },
      section/.style={
        draw=outflowred,
        line width=1.55pt
      },
      streamline/.style={
        draw=outflowblue!62,
        line width=0.8pt,
        postaction={decorate},
        decoration={
          markings,
          mark=at position 0.60 with
          {\arrow{Stealth[length=1.8mm]}}
        }
      }
    ]

    \def\rin{1.62}
    \def\rout{4.55}
    \def\anglow{-18}
    \def\anghigh{42}

    \path[tubefill]
      (\anglow:\rin)
      .. controls (2.27,-0.48) and (3.30,-0.63) ..
      (\anglow+4:\rout)
      arc[
        start angle=\anglow+4,
        end angle=\anghigh-5,
        radius=\rout
      ]
      .. controls (3.26,2.69) and (2.43,1.58) ..
      (\anghigh:\rin)
      arc[
        start angle=\anghigh,
        end angle=\anglow,
        radius=\rin
      ]
      -- cycle;

    \draw[circleline] (0,0) circle (\rin);
    \draw[circleline] (0,0) circle (\rout);

\draw[side]
  (\anglow:\rin)
  .. controls (2.27,-0.48) and (3.30,-0.63) ..
  (\anglow+4:\rout)
  node[pos=0.5,font=\small,below=8pt] {$\Gamma_-^R$};

\draw[side]
  (\anghigh:\rin)
  .. controls (2.43,1.58) and (3.26,2.69) ..
  (\anghigh-5:\rout)
  node[pos=0.5,font=\small,above=5pt] {$\Gamma_+^R$};
    \draw[section]
      (\anglow:\rin)
      arc[
        start angle=\anglow,
        end angle=\anghigh,
        radius=\rin
      ];

    \draw[section]
      (\anglow+4:\rout)
      arc[
        start angle=\anglow+4,
        end angle=\anghigh-5,
        radius=\rout
      ];

    \foreach \a/\b in {0/5,13/17,27/29}{
      \draw[streamline]
        (\a:\rin)
        .. controls
          (2.35,{0.36+0.035*\a})
          and
          (3.33,{0.22+0.055*\a})
        ..
        (\b:\rout);
    }

    \node[font=\small] at (1.02,0.66) {$I_{R_1}$};
    \node[font=\small] at (4.88,0.94) {$I_R$};
    \node[font=\small] at (2.72,0.88) {$T_R$};

    \node[font=\scriptsize,black!60]
      at (-0.22,-1.88) {$r=R_1$};

    \node[font=\scriptsize,black!60]
      at (-0.25,-4.82) {$r=R$};

  \end{tikzpicture}

\caption{The truncated stream tube $T_R$ generated by the arc
$I_{R_1}$. Its boundary consists of the circular sections $I_{R_1}$ and
$I_R$, shown in red, and the streamline sides $\Gamma_-^R$ and
$\Gamma_+^R$, shown in blue. The lighter curves represent trajectories
inside the tube.}
  \label{fig:stream-tube-geometry}
\end{figure}
\medskip
\noindent
\underline{\emph{(d) Conservation of flux.}}
For $\rho\geq R_1$, set
\begin{equation}\label{eq:tube-flux}
  {\Phi_T}(\rho):=\int_{I_\rho}v\cdot\er{\dd\ell}.
\end{equation}

Fix $R_1\leq\rho_1<\rho_2$ and set
\[
  T_{\rho_1,\rho_2}
  :=F\bigl((s_-,s_+)\times(\rho_1,\rho_2)\bigr).
\]
The same argument as in part (c) shows that
$T_{\rho_1,\rho_2}$ is a bounded piecewise $C^1$ Jordan
domain. Its outward unit normal is $\er$ on $I_{\rho_2}$
and $-\er$ on $I_{\rho_1}$. The velocity is tangent to
the two trajectory segments joining these sections.
The divergence theorem therefore gives
\[
  0=\int_{T_{\rho_1,\rho_2}}\diver v\dd x
   =\Phi_T(\rho_2)-\Phi_T(\rho_1).
\]
Thus ${\Phi_T}(\rho)={\Phi_T}(R_1)=:{\Phi_T}$. Since $v_r>0$ on the nondegenerate arc
$I_{R_1}$,
\begin{equation}\label{eq:positive-tube-flux}
  {\Phi_T}=\int_{I_{R_1}}v_r{\dd\ell}>0.
\end{equation}

\medskip
\noindent
\underline{\emph{(e) Quadratic growth of the tube.}}
Let
$\vartheta(\rho):=\mathcal H^1(I_\rho)/\rho$.
By \eqref{eq:speed-upper},
\[
  {\Phi_T}
  =\int_{I_\rho}v_r{\dd\ell}
  \leq\int_{I_\rho}|v|{\dd\ell}
  \leq\frac{C}{\rho}\mathcal H^1(I_\rho)
  =C\vartheta(\rho).
\]
Therefore,
\begin{equation}\label{eq:angular-width-lower}
  \vartheta(\rho)\geq\frac {{\Phi_T}}C
  \ \textup{for }\rho\geq R_1.
\end{equation}
The function $r(x)=|x|$ is Lipschitz on $\overline{T_R}$ and satisfies
$|\nabla r|=1$ there. Moreover,
$T_R\cap\partial B_\rho$ is the relative interior of $I_\rho$ for
$R_1<\rho<R$. Hence the coarea formula gives
\[
  \mathcal L^2(T_R)
  =\int_{R_1}^R\mathcal H^1(I_\rho){\dd\rho}
  =\int_{R_1}^R\rho\vartheta(\rho){\dd\rho}.
\]
Consequently,
\begin{equation}\label{eq:tube-area-lower}
  \mathcal L^2(T_R)
  \geq\frac{{\Phi_T}}{2C}(R^2-R_1^2).
\end{equation}

\medskip
\noindent
\underline{\emph{(f) Quadratic lower bound for the total vorticity.}}
Every $y\in T_R$ can be written as
$y=X(\tau(s,\rho),\gamma_1(s))$ {for $\rho=|y|$ and some $s\in (s_-,s_+)$}. 
{{By Lemma} \ref{lem:vorticity-transport}} and
\eqref{eq:omega-on-initial-arc},
\begin{equation}\label{eq:vorticity-sign-in-tube}
  \sigma\omega(y)\geq\delta
  \ \textup{for }y\in T_R.
\end{equation}
Combining this with \eqref{eq:tube-area-lower}, we obtain
\begin{equation}\label{eq:vorticity-lower}
  \left|\int_{T_R}\omega{\dd y}\right|
  \geq\delta\mathcal L^2(T_R)
  \geq\frac{\delta \Phi_T}{2C}(R^2-R_1^2).
\end{equation}

\medskip
\noindent
\underline{\emph{(g) Logarithmic upper bound from Stokes' theorem.}}
The field $v$ is $C^1$ on a neighborhood of $\overline{T_R}$, its curl is
the continuous function $\omega$, and $\partial T_R$ is piecewise $C^1$.
Let $\mathbf t_{\partial T_R}$ denote the positively oriented unit tangent to
$\partial T_R$, defined on each regular boundary arc. The 
Green--Stokes theorem applies and gives
\begin{equation}\label{eq:stokes-on-tube}
  \int_{T_R}\omega{\dd y}
  =
  \int_{\partial T_R}v\cdot\mathbf t_{\partial T_R}{\dd\ell}.
\end{equation}
For $\rho\in\{R_1,R\}$, the map $s\mapsto F(s,\rho)$ is a regular injective parametrization of a
compact interval into $\partial B_\rho$. Its image is therefore a simple
proper circular arc, and
$\mathcal H^1(I_\rho)\leq2\pi\rho$. It follows that
\[
  \left|\int_{I_\rho}
  v\cdot\mathbf t_{\partial T_R}{\dd\ell}\right|
  \leq\frac{C}{\rho}\mathcal H^1(I_\rho)
  \leq2\pi C.
\]
Thus the two transverse arcs contribute at most $4\pi C$. Consider either streamline side and write
$x_\pm(r):=F(s_\pm,r)$, $R_1\leq r\leq R$. Formula
\eqref{eq:rho-direction} gives
\[
  x_\pm'(r)=\frac{v(x_\pm(r))}{v_r(x_\pm(r))},
  \quad
  \dd\ell=\frac{|v(x_\pm(r))|}{v_r(x_\pm(r))}{\dd r}.
\]
Consequently,
\begin{equation}\label{eq:streamline-side-identity}
  \int_{\Gamma_\pm^R}|v|{\dd\ell}
  =
  \int_{R_1}^R
  \frac{|v(x_\pm(r))|^2}{v_r(x_\pm(r))}{\dd r}.
\end{equation}
Condition \eqref{eq:far-field-cost} yields
\begin{equation}\label{eq:streamline-side-bound}
  \int_{\Gamma_\pm^R}|v|{\dd\ell}
  \leq C\int_{R_1}^R\frac{\dd r}{r}
  =C\log\frac R{R_1}.
\end{equation}
There are two streamline sides. Therefore,
\begin{equation}\label{eq:circulation-upper}
  \left|\int_{T_R}\omega{\dd y}\right|
  \leq4\pi C+2C\log\frac R{R_1}.
\end{equation}

\medskip
\noindent
\underline{\emph{(h) Contradiction.}}
Estimates \eqref{eq:vorticity-lower} and
\eqref{eq:circulation-upper} imply
\[
  \frac{\delta \Phi_T}{2C}(R^2-R_1^2)
  \leq4\pi C+2C\log\frac R{R_1}.
\]
The constants $R_1$, $\delta$, $\Phi_T$ and $C$ are
independent of $R$. After division by $R^2$ and passage to the limit $R\to\infty$, we obtain
$\delta \Phi_T/(2C)\leq0$, contradicting $\delta>0$ and
$\Phi_T>0$. Hence
$\omega\equiv0$ in $\Om$. This proves implication \ref{itemi} $\Rightarrow$ \ref{itemii} and completes
the proof of Theorem~\ref{thm:main}.
\section*{Acknowledgments}
D. Alonso-Or\'an acknowledges partial support from grant RYC2023-045563-I,
funded by
MICIU/\allowbreak{}AEI/\allowbreak{}10.13039/\allowbreak{}501100011033
and ESF+.
D. Alonso-Or\'an also acknowledges support from the projects
``EDPs Deterministas y Estoc\'asticas en Biolog\'ia y Mec\'anica
de Fluidos'' (PID2025-171734NAI00) and
``An\'alisis Matem\'atico Aplicado y Ecuaciones Diferenciales''
(PID2022-141187NB-I00), funded by
MCIN/\allowbreak{}AEI/\allowbreak{}10.13039/\allowbreak{}501100011033/\allowbreak{}FEDER, UE.

D. Sánchez-Simón del Pino and J. J. L. Velázquez have been also funded by the Deutsche Forschungsgemeinschaft (DFG, German Research Foundation) under Germany’s Excellence Strategy - GZ 2047/1, Projekt-ID 390685813.

D. Sánchez-Simón del Pino further acknowledges the support of the Bonn International Graduate School of Mathematics (BIGS) at the Hausdorff Centre for Mathematics.

\end{document}